\documentclass[11pt,a4paper,reqno]{article}

\usepackage{graphicx}
\usepackage{amsmath}
\usepackage{color}

\usepackage{textcomp}
\usepackage{eurosym}

\usepackage{amsthm}
\usepackage{cite}

\newtheorem{lemma}{Lemma}[section]
\newtheorem{theorem}{Theorem}[section]

\newtheorem{property}{Property}[section]

\newtheorem{definition}{Definition}[section]
\newtheorem{remark}{Remark}[section]

\newcommand{\RRe}{\operatorname{Re}}

\begin{document}

\title{\textbf{Cyclic Averages of Regular Polygons and \\ Platonic Solids}}
\author{\textbf{Mamuka Meskhishvili}}

\date{}
\maketitle

\begin{abstract}
The concept of the cyclic averages are introduced for a regular polygon $P_n$ and a Platonic solid $T_n$. It is shown that cyclic averages of equal powers are the same for various $P_n(T_n)$, but their number is characteristic of $P_n(T_n)$. Given the definition of a circle (sphere) by the vertices of $P_n(T_n)$ and on the base of the cyclic averages are established the common metrical relations of $P_n(T_n)$.


\end{abstract}

\bigskip
\bigskip

\section{Introduction}
\label{sec:1}

Consider a finite set of $n$ points in the plane (space), then locus of points such that the sum of the squares of distances to the given points is constant, is a circle (sphere), whose center is at the centroid of the given points \cite{1, 2}.

Denote by $M(d_1,d_2,\dots,d_n,L)$ an arbitrary point in the plane (space) of a regular polygon (Platonic solid) of distances $d_1,d_2,\dots,d_n$ to the vertices $A_1,A_2,\dots,A_n$, then:
\begin{equation}\label{eq:*}
    \sum_{1}^n d_i^2=n(R^2+L^2),  \tag{$*$}
\end{equation}
where $R$ is the radius of circumscribed circle (sphere) of the regular polygon (Platonic solid) and $L$ is the distance between the point $M$ and the centroid $O$.

The symmetric equation exists for an equilateral triangle and an arbitrary point   \linebreak    $M(d_1,d_2,d_3,L)$ in the plane of the triangle
\begin{equation}\label{eq:**}
    3(d_1^4+d_2^4+d_3^4+a^4)=(d_1^2+d_2^2+d_3^2+a^2)^2,    \tag{$**$}
\end{equation}
where $a$ is the side of the triangle \cite{3, 4}.

The arbitrary point is always considered in the plane of the regular polygon, and in the space of the Platonic solid, respectively.

From relations \eqref{eq:*} and \eqref{eq:**} follows:
\begin{align*}
    \sum_{1}^3 d_i^2 & =3(R^2+L^2), \\
    \sum_{1}^3 d_i^4 & =3\big((R^2+L^2)^2+2R^2L^2\big).
\end{align*}
For a given equilateral triangle, the side $a$ as well as the circumradius $R$ are fixed so that,

\begin{theorem}\label{th:1.1}
The locus of points such that
$$  \sum_{1}^3 d_i^4=const       $$
is a circle, center of which is the centroid.
\end{theorem}

As we see, the distances are considered to the second and the fourth powers. Naturally, we are interested to know what happens if we consider the distances of higher powers.

\bigskip
\section{Preliminaries}
\label{sec:2}

For an equilateral triangle the expression $\sum\limits_{1}^3 d_i^6$, contains $\alpha$ -- the angle between $R$ and $L$, so the locus is not a circle, but for a square case the answer is surprising: the locus of points such that
$$  \sum_{1}^4 d_i^6=const        $$
is a circle.

Generally, the locus is a circle (sphere) if and only if the sum of power distances can be expressed in terms of $L$ and some fixed element (with length) of a given regular polygon (Platonic solid). The fixed element is possible to express in terms of $R$, so we denote such sums by the symbol $\sum_{[R,L]}$, or $\sum_{[R,L]}^{(2m)}$ -- to indicate the like powers of the distances.

Denote by $P_n(R)$ and $T_n(R)$ a regular polygon and Platonic solid, respectively, with an $n$ number of the vertices and circumscribed radius $R$. The value of the $\sum_{[R,L]}$ remains constant when the point $M$ moves on the circle $C(O,L)$  (sphere $S(O,L))$. So,

\begin{definition}\label{def:2.1}
$\sum_{[R,L]}$ -- is the sum of like powers of the distances $d_1,\dots,d_n$ from an arbitrary point $M(d_1,\dots,d_n,L)$ to the vertices $P_n(R)$ $(T_n(R))$ the value of which is constant for any point of the $C(O,L)$ $(S(O,L))$.
\end{definition}

It is clear, the sum of odd power contains radicals and never will be $\sum_{[R,L]}$.

For establishing common properties of the $P_n$ and $T_n$ discussing average of $\sum_{[R,L]}$ is much preferred
$$  S_n^{(2m)}=\frac{1}{n}\sum\nolimits_{[R,L]}^{(2m)}.        $$

\begin{definition}\label{def:2.2}
The cyclic averages $S_n^{(2m)}$ $(S_{[n]}^{(2m)})$ of a regular polygon (Platonic solid) is the average of the sum $\sum_{[R,L]}^{(2m)}$.
\end{definition}

We call such averages the cyclic averages, because as we prove the cyclic averages of equal powers of various $P_n$ $(T_n)$ for fixed $R$ and $L$ are the same (if they exist):
$$  S_{n_1}^{(2m)}=S_{n_2}^{(2m)},      $$
if $n_1\leq n_2$.

On the other hand for any given $P_n$ the number of $S_n^{(2m)}$ (as well as $\sum_{[R,L]}^{(2m)})$ is defined uniquely, so the number of the cyclic averages is characteristic of the regular polygon.

For example, 2 cyclic averages exist for a regular 3-gon:
$$  S_3^{(2)} \;\;\text{and}\;\; S_3^{(4)},     $$
while for a regular 4-gon -- 3 cyclic averages:
$$  S_4^{(2)},\;S_4^{(4)} \;\;\text{and}\;\; S_4^{(6)}.     $$
They are in relations:
$$  S_3^{(2)}=S_4^{(2)} \;\;\text{and}\;\; S_3^{(4)}=S_4^{(4)}.     $$

To demonstrate the efficiency of cyclic averages the analogue of the relation \eqref{eq:**} will be obtained for the square. Firstly, we turn \eqref{eq:**} in terms of $R$ and the cyclic averages -- $S_3^{(2)}$, $S_3^{(4)}$:
$$  \frac{d_1^4+d_2^4+d_3^4}{3}+3R^4=\Big(\frac{d_1^2+d_2^2+d_3^2}{3}+R^2\Big)^2,       $$
then replace with
$$  S_3^{(2)}=S_4^{(2)}, \;\; S_3^{(4)}=S_4^{(4)} \;\;\text{and}\;\; R=\frac{a}{\sqrt{2}}\,;        $$
we get

\begin{theorem}\label{th:2.1}
For an arbitrary point $M(d_1,\dots,d_4,L)$ in the plane of a square:
$$  4(d_1^4+d_2^4+d_3^4+d_4^4+3a^4)=(d_1^2+d_2^2+d_3^2+d_4^2+2a^2)^2,       $$
where $a$ is the side of the square.
\end{theorem}

\bigskip
\section{Circle as Locus of Constant $\sum_{[R,L]}$ Sums}
\label{sec:3}

\begin{theorem}\label{th:3.1}
For an arbitrary point $M(d_1,d_2,\dots,d_n,L)$ in the plane of regular polygon $P_n(R)$:
$$  \sum_{i=1}^n d_i^{2m}=n\bigg[(R^2+L^2)^m+\sum_{k=1}^{\lfloor\frac{m}{2}\rfloor} \binom{m}{2k}\binom{2k}{k} R^{2k}L^{2k}(R^2+L^2)^{m-2k}\bigg],     $$
where $m=1,\dots,n-1$.
\end{theorem}

First we need to prove two lemmas.

\begin{lemma}\label{lem:3.1}
For arbitrary positive integers $m$ and $n$, such that $m<n$, the following condition
$$  \sum_{k=1}^n \cos\bigg(m\Big(\alpha-(k-1)\,\frac{2\pi}{n}\Big)\bigg)=0        $$
is satisfied, where $\alpha$ is an arbitrary angle.
\end{lemma}

Denote
$$  T=e^{im\alpha}+e^{im(\alpha-\frac{2\pi}{n})}+e^{im(\alpha-2\,\frac{2\pi}{n})}+\cdots+e^{im(\alpha-(n-1)\,\frac{2\pi}{n})}.     $$
The real part of $T$ is
$$  \RRe(T)=\sum_{k=1}^n \cos\bigg(m\Big(\alpha-(k-1)\,\frac{2\pi}{n}\Big)\bigg).      $$

The formula of the sum of geometric progression gives
\begin{gather*}
    T=e^{im\alpha}\bigg(1+e^{-im\,\frac{2\pi}{n}}+\big(e^{-im\,\frac{2\pi}{n}}\big)^2+\cdots+\big(e^{-im\,\frac{2\pi}{n}}\big)^{n-1}\bigg)=
                e^{im\alpha}\,\frac{1-(e^{-im\,\frac{2\pi}{n}})^n}{1-e^{-im\,\frac{2\pi}{n}}}\,, \\[0.5cm]
    e^{-im\,2\pi}=\cos(-2\pi m)+i\sin(-2\pi m)=1.
\end{gather*}

Since $m<n$, $e^{-im\,\frac{2\pi}{n}}\neq 1$. So $T=0$, i.e. $\RRe(T)=0$, which proves Lemma \ref{lem:3.1}.

\begin{remark}\label{rem:3.1}
If $m\geq n$, the sum always contains $\alpha$.
\end{remark}

\begin{lemma}\label{lem:3.2}
For arbitrary positive integers $m$ and $n$, such that $m<n$ and for an arbitrary angle $\alpha$ the following conditions are satisfied:
\begin{enumerate}
\item[] if $m$ is odd

$$  \sum_{k=1}^n \cos^m\Big(\alpha-(k-1)\,\frac{2\pi}{n}\Big)=0;     $$

\item[] if $m$ is even
$$  \sum_{k=1}^n \cos^m\Big(\alpha-(k-1)\,\frac{2\pi}{n}\Big)=n\,\frac{\binom{m}{\frac{m}{2}}}{2^m}\,.     $$
\end{enumerate}
\end{lemma}

When $m$ is odd, using the power-reduction formula for cosine
$$  \cos^m\theta=\frac{2}{2^m} \sum_{k=0}^{\frac{m-1}{2}} \binom{m}{k} \cos\big((m-2k)\theta\big),        $$
\allowdisplaybreaks
we obtain
\begin{align*}
    \sum_{k=1}^n \cos^m & \Big(\alpha-(k-1)\,\frac{2\pi}{n}\Big) \\[0.2cm]
    & =\cos^m\alpha+\cos^m\Big(\alpha-\frac{2\pi}{n}\Big)+\cdots+\cos^m\Big(\alpha-(n-1)\,\frac{2\pi}{n}\Big) \\[0.2cm]
    & =\frac{2}{2^m}\,\Bigg[\binom{m}{0}\cos m\alpha+\binom{m}{1}\cos (m-2)\alpha+\cdots+\binom{m}{\frac{m-1}{2}}\cos\alpha \\[0.2cm]
    &\quad +\binom{m}{0}\cos m\Big(\alpha-\frac{2\pi}{n}\Big)+\binom{m}{1}\cos(m-2)\Big(\alpha-\frac{2\pi}{n}\Big)+\cdots \\[0.2cm]
    &\qquad\qquad\qquad\qquad +\binom{m}{\frac{m-1}{2}}\cos\Big(\alpha-\frac{2\pi}{n}\Big)+\;\cdots\; \\[0.2cm]
    &\quad +\binom{m}{0}\cos m\Big(\alpha\!-\!(n\!-\!1)\,\frac{2\pi}{n}\Big)\!+\!\binom{m}{1}\cos(m\!-\!2)\Big(\alpha\!-\!(n\!-\!1)\,\frac{2\pi}{n}\Big) \\[0.2cm]
    &\qquad\qquad\quad\;\;\; +\cdots+\binom{m}{\frac{m-1}{2}}\cos\Big(\alpha-(n-1)\,\frac{2\pi}{n}\Big)\Bigg] \\[0.2cm]
    & =\frac{2}{2^m}\,\Bigg[\binom{m}{0}\bigg(\cos m\alpha+\cos m\Big(\alpha-\frac{2\pi}{n}\Big)+\cdots+\cos m\Big(\alpha-(n-1)\,\frac{2\pi}{n}\Big)\bigg) \\[0.2cm]
    &\qquad +\binom{m}{1}\bigg(\cos(m\!-\!2)\alpha\!+\!\cos(m\!-\!2)\Big(\alpha\!-\!\frac{2\pi}{n}\Big)\!+\cdots \\[0.2cm]
    &\qquad\qquad\qquad\qquad\qquad +\cos(m-2)\Big(\alpha-(n-1)\,\frac{2\pi}{n}\Big)\bigg)+\cdots \\[0.2cm]
    &\qquad +\binom{m}{\frac{m-1}{2}}\bigg(\cos\alpha+\cos\Big(\alpha-\frac{2\pi}{n}\Big)+\cdots+\cos\Big(\alpha-(n-1)\,\frac{2\pi}{n}\Big)\bigg)\Bigg].
\end{align*}
Since $m<n$, from Lemma~\ref{lem:3.1} it follows that each sum equals zero, which proves the first part of Lemma \ref{lem:3.2}.

When $m$ is even, the power-reduction formula for cosine is
$$  \cos^m\theta=\frac{1}{2^m}\,\binom{m}{\frac{m}{2}}+\frac{2}{2^m} \sum_{k=0}^{\frac{m}{2}-1}\binom{m}{k}\cos\big((m-2k)\theta\big).       $$
Analogously to the case with odd $m$, the sum of the second addenda vanishes, and since the number of the first addenda is $n$, the total sum equals
$$  n\,\binom{m}{\frac{m}{2}}{2^m}\,,     $$
which proves Lemma \ref{lem:3.2}.

\begin{proof}[Proof of Theorem \ref{th:3.1}]
We introduce the new notations
$$  A=R^2+L^2 \;\;\text{and}\;\; B=2RL.     $$
Then
\begin{multline*}
    \sum_{i=1}^n d_i^{2m}=(A-B\cos\alpha)^m+\bigg(A-B\cos\Big(\frac{2\pi}{n}-\alpha\Big)\bigg)^m \\
    +\bigg(A-B\cos\Big(2\cdot\frac{2\pi}{n}-\alpha\Big)\bigg)^m+\cdots+\bigg(A-B\cos\Big((n-1)\,\frac{2\pi}{n}-\alpha\Big)\bigg)^m.
\end{multline*}
If $m=1$, by Lemma~\ref{lem:3.1} we have
\begin{multline*}
    \sum_{i=1}^n d_i^{2m}=(A-B\cos\alpha) \\
    +\bigg(A-B\cos\Big(\frac{2\pi}{n}-\alpha\Big)\bigg)+\cdots+\bigg(A-B\cos\Big((n-1)\cdot\frac{2\pi}{n}-\alpha\Big)\bigg)=nA.
\end{multline*}
Therefore
$$  S_n^{(2)}=R^2+L^2.      $$
If $m>1$, we have
\allowdisplaybreaks[0]
\begin{align*}
    \sum_{i=1}^n d_i^{2m} & =nA^m-\binom{m}{1}A^{m-1}B\bigg(\cos\alpha+\cos\Big(\frac{2\pi}{n}-\alpha\Big)+\cdots+\cos\Big((n-1)\,\frac{2\pi}{n}-\alpha\Big)\bigg) \\[0.2cm]
    &\qquad +\binom{m}{2} A^{m-2}B^2\bigg(\cos^2\alpha+\cos^2\Big(\frac{2\pi}{n}-\alpha\Big)+\cdots+\cos^2\Big((n-1)\,\frac{2\pi}{n}-\alpha\Big)\bigg) \\[0.2cm]
    &\qquad -\binom{m}{3}A^{m-3}B^3\bigg(\cos^3\alpha+\cos^3\Big(\frac{2\pi}{n}-\alpha\Big)+\cdots+\cos^3\Big((n-1)\,\frac{2\pi}{n}-\alpha\Big)\bigg)+\cdots \\[0.2cm]
    &\qquad \pm\binom{m}{m}B^m\bigg(\cos^m\alpha+\cos^m\Big(\frac{2\pi}{n}-\alpha\Big)+\cdots+\cos^m\Big((n-1)\,\frac{2\pi}{n}-\alpha\Big)\bigg).
\end{align*}
According to Lemma \ref{lem:3.2}, all sums with the negative sign vanish and only the sums with the positive sign remain.

If $m$ is even
\begin{align*}
    \sum_{i=1}^n d_i^{2m} & =nA^m+\binom{m}{2}A^{m-2}B^2\bigg(\cos^2\alpha+\cos^2\Big(\frac{2\pi}{n}-\alpha\Big)+\cdots+\cos^2\Big((n-1)\,\frac{2\pi}{n}-\alpha\Big)\bigg)+\cdots\\[0.2cm]
    &\qquad\quad\;\, +\binom{m}{m}B^m\bigg(\cos^m\alpha+\cos^m\Big(\frac{2\pi}{n}-\alpha\Big)+\cdots+\cos^m\Big((n-1)\,\frac{2\pi}{n}-\alpha\Big)\bigg) \\[0.2cm]
    & =n\Bigg(A^m+\sum_{k=1}^{\frac{m}{2}} \binom{m}{2k}A^{m-2k}B^{2k}\,\frac{1}{2^{2k}} \binom{2k}{k}\Bigg).
\end{align*}
If $m$ is odd
\begin{align*}
    \sum_{i=1}^n d_i^{2m} & =nA^m+\binom{m}{2}A^{m-2}B^2\bigg(\cos^2\alpha+\cos^2\Big(\frac{2\pi}{n}-\alpha\Big)+\cdots+\cos^2\Big((n-1)\,\frac{2\pi}{n}-\alpha\Big)\bigg)+\cdots \\[0.2cm]
    &\qquad\quad\;\, +\binom{m}{m-1}AB^{m-1}\bigg(\cos^{m-1}\alpha+\cos^{m-1}\Big(\frac{2\pi}{n}-\alpha\Big) \\[0.2cm]
    &\qquad\qquad\qquad\qquad\qquad\qquad\qquad\qquad\qquad +\cdots+\cos^{m-1}\Big((n-1)\,\frac{2\pi}{n}-\alpha\Big)\bigg) \\[0.2cm]
    & =n\Bigg(A^m+\sum_{k=1}^{\frac{m-1}{2}} \binom{m}{2k} A^{m-2k}B^{2k}\,\frac{1}{2^{2k}}\binom{2k}{k}\Bigg).
\end{align*}

Using the floor function (the integer part), the obtained results can be combined into a single formula as follows
$$  \sum_{i=1}^n d_i^{2m}=n\bigg(A^m+\sum_{k=1}^{\lfloor\frac{m}{2}\rfloor} \binom{m}{2k} A^{m-2k}B^{2k}\,\frac{1}{2^{2k}}\,\binom{2k}{k}\bigg),     $$
which proves the theorem.
\end{proof}

From Theorem \ref{th:3.1} each sums
$$  \sum_{i=1}^n d_i^{2m}, \;\;\text{where}\;\; m=1,2,\dots,n-1     $$
are the $\sum_{[R,L]}$ sums. Beginning from the $m\geq n$ all sums of power distances contain $\alpha$    \linebreak    (Remark~\ref{rem:3.1}).

For example for $P_3$ the sums contain:
\begin{enumerate}
\item[-] $\cos 3\alpha$, if $m=3,4,5$;

\item[-] $\cos 3\alpha$ and $\cos 6\alpha$, if $m=6,7,8$;

\item[-] $\cos 3\alpha$, $\cos 6\alpha$ and $\cos 9\alpha$, if $m=9,10,11$.
\end{enumerate}

Generally for $m\geq n$ the sums $\sum\limits_{i=1}^n d_i^{2m}$ contain cosine of the multiples of $n\alpha$. The study of such sums is beyond the scope of this article.

Therefore for $P_n$ exist an $n-1$ number of $\sum_{[R,L]}$ sums and if they are constant the locus for each case is a circle:

\begin{theorem}\label{th:3.2}
The locus of points such that the sum of the $(2m)$-th power of the distances to the vertices of a given $P_n(R)$ is constant is a circle, if
$$  \sum_{i=1}^n d_i^{2m}>nR^{2m}, \;\;\text{where}\;\; m=1,2,\dots,n-1,        $$
whose center is the centroid of the $P_n(R)$.
\end{theorem}

\begin{remark}\label{rem:3.2}
\ \ \
\begin{enumerate}
\item[-] If $\sum\limits_{i=1}^n d_i^{2m}=nR^{2m}$, the locus is the centroid of the polygon.

\item[-] If $\sum\limits_{i=1}^n d_i^{2m}<nR^{2m}$, the locus is the empty set.
\end{enumerate}
\end{remark}

\bigskip
\section{Cyclic Averages of Regular Polygons}
\label{sec:4}

The properties of the cyclic average are as follows:

\begin{property}\label{pr:4.1}
Each regular $n$-gon has an $n-1$ number of cyclic averages
$$  S_n^{(2)},S_n^{(4)},\dots,S_n^{(2n-2)}.     $$
\end{property}

\begin{property}\label{pr:4.2}
For fixed $R$ and $L$, the cyclic averages of equal powers of various regular $n$-gons are the same:
\begin{align*}
    S_3^{(2)}=S_4^{(2)}=S_5^{(2)} & =S_6^{(2)}=\cdots, \\[0.2cm]
    S_3^{(4)}=S_4^{(4)}=S_5^{(4)} & =S_6^{(4)}=\cdots, \\[0.2cm]
    S_4^{(6)}=S_5^{(6)} & =S_6^{(6)}=\cdots, \\[0.2cm]
    S_5^{(8)} & =S_6^{(8)}=\cdots\,.
\end{align*}
\end{property}

\begin{property}\label{pr:4.3}
Any relations in terms of the cyclic averages $S_{n_1}^{(2m)}$, the circumscibed radius $R$ and the distance $L$, which are satisfied for a regular $n_1$-gon, are at the same time satisfied for any regular $n_2$-gon, where $n_1\leq n_2$, i.e. $S_{n_1}^{(2m)}$ can be replaced by $S_{n_2}^{(2m)}$.
\end{property}

Eliminate $L$ from the relations of Theorem \ref{th:3.1} we obtain:

\begin{theorem}\label{th:4.1}
For any regular $n$-gon:
$$  S_n^{(2m)}=(S_n^{(2)})^m+\sum_{k=1}^{\lfloor\frac{m}{2}\rfloor} \binom{m}{2k}\binom{2k}{k} R^{2k}(S_n^{(2)}-R^2)^k(S_n^{(2)})^{m-2k},       $$
where $m=2,\dots,n-1$.
\end{theorem}

In terms of $S_n^{(2)}$ and $S_n^{(4)}$:

\begin{theorem}\label{th:4.2}
For any regular $n$-gon:
$$  S_n^{(2m)}=(S_n^{(2)})^m+\sum_{k=1}^{\lfloor\frac{m}{2}\rfloor} \frac{1}{2^k}\,\binom{m}{2k}\binom{2k}{k} \big(S_n^{(4)}-(S_n^{(2)})^2\big)^k(S_n^{(2)})^{m-2k},       $$
where $m=3,\dots,n-1$.
\end{theorem}

The first two relations of Theorem \ref{th:3.1} imply:

\begin{theorem}\label{th:4.3}
For any regular $n$-gon:
\begin{align*}
    R^2 & =\frac{1}{2}\,\Big(S_n^{(2)}\pm\sqrt{3(S_n^{(2)})^2-2S_n^{(4)}}\Big), \\[0.2cm]
    L^2 & =\frac{1}{2}\,\Big(S_n^{(2)}\mp\sqrt{3(S_n^{(2)})^2-2S_n^{(4)}}\Big).
\end{align*}
\end{theorem}

The points on the circumscribed circle satisfy
$$  3(S_n^{(2)})^2=2S_n^{(4)},          $$
so

\begin{theorem}\label{th:4.4}
For any point on the circumscribed circle of the regular $n$-gon:
$$  3\Big(\sum_{i=1}^n d_i^2\Big)^2=2n\sum_{i=1}^n d_i^4.     $$
\end{theorem}

\subsection{Equilateral triangle}
\label{subsec:4.1}
\medskip

There are $2$ cyclic averages:
\begin{align*}
    S_3^{(2)} & =\frac{1}{3}\,(d_1^2+d_2^2+d_3^2)=R^2+L^2, \\[0.2cm]
    S_3^{(4)} & =\frac{1}{3}\,(d_1^4+d_2^4+d_3^4)=(R^2+L^2)^2+2R^2L^2.
\end{align*}
In general case from Theorem \ref{th:4.1}, for $n\geq 3$ \cite{5}
$$  S_n^{(4)}+3R^4=(S_n^{(2)}+R^2)^2.       $$

Denote by the symbol -- $\triangle_{(a,b,c)}$ the area of a triangle whose sides have lengths $a$, $b$, $c$. Then solution of the system of the cyclic averages is:

\begin{theorem}\label{th:4.5}
For any point $M(d_1,d_2,d_3,L)$ and $P_3(R)$
\begin{align*}
    d_1 & =d_1, \\
    d_2^2 & =\frac{1}{2}\,\Big(3(R^2+L^2)-d_1^2\pm 4\sqrt{3}\triangle_{(R,L,d_1)}\Big), \\[0.2cm]
    d_3^2 & =\frac{1}{2}\,\Big(3(R^2+L^2)-d_1^2\mp 4\sqrt{3}\triangle_{(R,L,d_1)}\Big).
\end{align*}
\end{theorem}

For $P_3$
$$  3(S_n^{(2)})^2-2S_n^{(4)}=\frac{1}{3}\,\Big((d_1^2+d_2^2+d_3^2)^2-2(d_1^4+d_2^4+d_3^4)\Big)=\frac{16}{3}\,\triangle_{(d_1,d_2,d_3)}^2,        $$
and
\begin{align*}
    R^2 & =\frac{1}{6}\,\Big(d_1^2+d_2^2+d_3^2\pm 4\sqrt{3}\triangle_{(d_1,d_2,d_3)}\Big), \\[0.2cm]
    L^2 & =\frac{1}{6}\,\Big(d_1^2+d_2^2+d_3^2\mp 4\sqrt{3}\triangle_{(d_1,d_2,d_3)}\Big).
\end{align*}

For any point on the circumscribed circle, follows the area -- $\triangle_{(d_1,d_2,d_3)}$ should be zero. Indeed for the largest distance $d_3=d_1+d_2$ holds.

\subsection{Square}
\label{subsec:4.2}

There are $3$ cyclic averages:
\allowdisplaybreaks[0]
\begin{align*}
    S_4^{(2)} & =\frac{1}{4}\,(d_1^2+d_2^2+d_3^2+d_4^2)=R^2+L^2, \\[0.2cm]
    S_4^{(4)} & =\frac{1}{4}\,(d_1^4+d_2^4+d_3^4+d_4^4)=(R^2+L^2)^2+2R^2L^2, \\[0.2cm]
    S_4^{(6)} & =\frac{1}{4}\,(d_1^6+d_2^6+d_3^6+d_4^6)=(R^2+L^2)^3+6R^2L^2(R^2+L^2).
\end{align*}

From Theorems \ref{th:4.1} and \ref{th:4.2}

\begin{theorem}\label{th:4.6}
For any regular $n$-gon, where $n\geq 4$:
\begin{align*}
    S_n^{(6)} & =S_n^{(2)}\big((S_n^{(2)}+3R^2)^2-15R^4\big), \\[0.2cm]
    S_n^{(6)} & =S_n^{(2)}\big(3S_n^{(4)}-2(S_n^{(2)})^2\big).
\end{align*}
\end{theorem}

From Theorem \ref{th:4.6} follows:
$$  8(d_1^6+d_2^6+d_3^6+d_4^6)+(d_1^2+d_2^2+d_3^2+d_4^2)^3=6(d_1^2+d_2^2+d_3^2+d_4^2)(d_1^4+d_2^4+d_3^4+d_4^4),       $$
which is equivalent to
$$  3(d_1^2+d_2^2-d_3^2-d_4^2)(d_1^2+d_3^2-d_2^2-d_4^2)(d_1^2+d_4^2-d_2^2-d_3^2)=0.      $$
So
$$  d_1^2+d_3^2=d_2^2+d_4^2     $$
holds.

Obtained relation has generalization for regular $n$-gon. If $n$ is even for the diametrically opposed vertices:

\begin{theorem}\label{th:4.2-3}
For any regular $n$-gon, with even number of vertices $n=2k$:
$$  d_1^2+d_{1+k}^2=d_2^2+d_{2+k}^2=\cdots=d_k^2+d_{2k}^2=2(R^2+L^2).       $$
\end{theorem}

Theorem \ref{th:4.2-3} simplifies the system of the cyclic averages:
\begin{align*}
    S_4^{(4)}+3R^2 & =(S_4^{(2)}+R^2)^2, \\[0.2cm]
    d_1^2+d_3^2 & =d_2^2+d_4^2;
\end{align*}
which is analogue to systems obtained in \cite{6, 7}. Moreover, in terms of $R$ and $L$, we get:
\begin{gather*}
    d_1^2+d_3^2=d_2^2+d_4^2=2(R^2+L^2), \\[0.2cm]
    d_1^2d_3^2+d_2^2d_4^2=2(R^4+L^4).
\end{gather*}
The solution of which is:

\begin{theorem}\label{th:4.2-4}
For any point $M(d_1,d_2,d_3,d_4,L)$ and $P_4(R)$:
\begin{align*}
    d_1 & =d_1, \\[0.2cm]
    d_2^2 & =R^2+L^2\pm 4\triangle_{(R,L,d_1)}, \\[0.2cm]
    d_3^2 & =2(R^2+L^2)-d_1^2, \\[0.2cm]
    d_4^2 & =R^2+L^2\mp 4\triangle_{(R,L,d_1)}.
\end{align*}
\end{theorem}

For $P_4$
%\allowdisplaybreaks[0]
\begin{align*}
    3(S_n^{(2)})^2-2S_n^{(2)} & =\frac{1}{16}\,\Big[3(d_1^2+d_2^2+d_3^2+d_4^2)^2-8(d_1^4+d_2^4+d_3^4+d_4^4)\Big] \\[0.2cm]
    & =4\triangle_{(d_1,\sqrt{2}\,d_2,d_3)}^2=4\triangle_{(d_2,\sqrt{2}\,d_3,d_4)}^2,
\end{align*}
and
\begin{gather*}
    R^2=\frac{1}{4}\,(d_1^2+d_3^2)\pm \triangle_{(d_1,\sqrt{2}\,d_2,d_3)}=\frac{1}{4}\,(d_2^2+d_4^2)\pm \triangle_{(d_2,\sqrt{2}\,d_3,d_4)}, \\[0.2cm]
    L^2=\frac{1}{4}\,(d_1^2+d_3^2)\mp \triangle_{(d_1,\sqrt{2}\,d_2,d_3)}=\frac{1}{4}\,(d_2^2+d_4^2)\mp \triangle_{(d_2,\sqrt{2}\,d_3,d_4)}.
\end{gather*}

For any point on the circumscribed circle the areas -- $\triangle_{(d_1,\sqrt{2}\,d_2,d_3)}$ and $\triangle_{(d_2,\sqrt{2}\,d_3,d_4)}$ should be zero. Indeed, if the point on the minor arc $A_1A_2$ are satisfied
$$  d_1+\sqrt{2}\,d_2=d_3 \;\;\text{and}\;\; d_2+d_4=\sqrt{2}\,d_3.      $$

\subsection{Regular Pentagon, Hexagon and Heptagon}
\label{subsec:4.3}

There are 4, 5 and 6 cyclic averages for the $P_5$, $P_6$ and $P_7$ cases, respectively:
\begin{align*}
    S_5^{(2)}=S_6^{(2)}=S_7^{(2)} & =R^2+L^2, \\[0.2cm]
    S_5^{(4)}=S_6^{(4)}=S_7^{(4)} & =(R^2+L^2)^2+2R^2L^2, \\[0.2cm]
    S_5^{(6)}=S_6^{(6)}=S_7^{(6)} & =(R^2+L^2)^3+6R^2L^2(R^2+L^2), \\[0.2cm]
    S_5^{(8)}=S_6^{(8)}=S_7^{(8)} & =(R^2+L^2)^4+12R^2L^2(R^2+L^2)^2+6R^4L^4, \\[0.2cm]
    S_6^{(10)}=S_7^{(10)} & =(R^2+L^2)^5+20R^2L^2(R^2+L^2)^3+30R^4L^4(R^2+L^2), \\[0.2cm]
    S_7^{(12)} & =(R^2+L^2)^6+30R^2L^2(R^2+L^2)^4+90R^4L^4(R^2+L^2)^2+20R^6L^6.
\end{align*}
These systems are simplified for the regular hexagon case only.

The vertices $A_1$, $A_3$, $A_5$ and $A_2$, $A_4$, $A_6$ form two equilateral triangles, so they satisfy two cyclic relations for $P_3$. Generally for $n$-gon if $n$ divisible by 3:

\begin{theorem}\label{th:4.3-1}
For any regular $n$-gon, if $n=3\ell$
\begin{align*}
    d_1^2+d_{1+\ell}^2+d_{1+2\ell}^2 & =\cdots=d_{\ell}^2+d_{2\ell}^2+d_{3\ell}^2=3(R^2+L^2), \\[0.2cm]
    d_1^4+d_{1+\ell}^4+d_{1+2\ell}^4 & =\cdots=d_{\ell}^4+d_{2\ell}^4+d_{3\ell}^4=3\big((R^2+L^2)^2+2R^2L^2\big).
\end{align*}
\end{theorem}

The Theorem \ref{th:4.2-3} and Theorem \ref{th:4.3-1} simplify the system of the cyclic averages for the regular hexagon:
\begin{align*}
    d_1^2+d_4^2=d_2^2+d_5^2=d_3^2+d_6^2 & =2(R^2+L^2), \\[0.2cm]
    d_1^2+d_3^2+d_5^2=d_2^2+d_4^2+d_6^2 & =3(R^2+L^2), \\[0.2cm]
    d_1^4+d_3^4+d_5^4=d_2^4+d_4^4+d_6^4 & =3\big((R^2+L^2)^2+2R^2L^2\big).
\end{align*}
By using these relations, we get explicit expressions for distances:

\begin{theorem}\label{th:4.3-2}
For any point $M(d_1,d_2,\dots,d_6,L)$ and $P_6(R)$:
\begin{align*}
    d_1 & =d_1, \\[0.2cm]
    d_2^2 & =\frac{1}{2}\,\Big(R^2+L^2+d_1^2\pm 4\sqrt{3}\,\triangle_{(R,L,d_1)}\Big), \\[0.2cm]
    d_3^2 & =\frac{1}{2}\,\Big(3R^2+3L^2-d_1^2\pm 4\sqrt{3}\,\triangle_{(R,L,d_1)}\Big), \\[0.2cm]
    d_4^2 & =2(R^2+L^2)-d_1^2, \\[0.2cm]
    d_5^2 & =\frac{1}{2}\,\Big(3R^2+3L^2-d_1^2\mp 4\sqrt{3}\,\triangle_{(R,L,d_1)}\Big), \\[0.2cm]
    d_6^2 & =\frac{1}{2}\,\Big(R^2+L^2+d_1^2\mp 4\sqrt{3}\,\triangle_{(R,L,d_1)}\Big).
\end{align*}
\end{theorem}

For $P_6$:
\begin{align*}
    3(S_n^{(2)})^2-2S_n^{(4)} & =3\Big(\frac{d_1^2+d_2^2+\cdots+d_6^2}{6}\Big)^2-2\,\frac{d_1^4+d_2^4+\cdots+d_6^4}{6} \\[0.2cm]
    & =\frac{1}{3}\,\big((d_1^2+d_3^2+d_5^2)^2-2(d_1^4+d_3^4+d_5^4)\big) \\[0.2cm]
    & =\frac{16}{3}\,\triangle_{(d_1,d_3,d_5)}^2=\frac{16}{3}\,\triangle_{(d_2,d_4,d_6)}^2
\end{align*}
and
\begin{align*}
    R^2 & =\frac{1}{6}\,\Big(d_1^2+d_3^2+d_5^2\pm 4\sqrt{3}\triangle_{(d_1,d_3,d_5)}\Big)=
                    \frac{1}{6}\,\Big(d_2^2+d_4^2+d_6^2\pm 4\sqrt{3}\triangle_{(d_2,d_4,d_6)}\Big), \\[0.2cm]
    L^2 & =\frac{1}{6}\,\Big(d_1^2+d_3^2+d_5^2\mp 4\sqrt{3}\triangle_{(d_1,d_3,d_5)}\Big)=
                    \frac{1}{6}\,\Big(d_2^2+d_4^2+d_6^2\mp 4\sqrt{3}\triangle_{(d_2,d_4,d_6)}\Big).
\end{align*}

For any point on the circumscribed circle the area $\triangle_{(d_1,d_3,d_5)}$ as well as $\triangle_{(d_2,d_4,d_6)}$ vanishes. Indeed if the point on the minor arc $A_1A_2$:
$$  d_1+d_3=d_5 \;\;\text{and}\;\; d_2+d_6=d_4.     $$

\subsection{Regular Octagon, Nonagon and Decagon}
\label{subsec:4.4}

There are 8, 9 and 10 cyclic averages for the $P_8$, $P_9$ and $P_{10}$ cases, respectively. The cyclic averages from the second to the twelfth powers are the same as for regular heptagon, so we write only new ones:
\allowdisplaybreaks[0]
\begin{align*}
    S_8^{(14)}=S_9^{(14)}=S_{10}^{(14)} & =(R^2+L^2)^7+42R^2L^2(R^2+L^2)^5+210R^4L^4(R^2+L^2)^3 \\[0.2cm]
    &\qquad\quad +140R^6L^6(R^2+L^2), \\[0.2cm]
    S_9^{(16)}=S_{10}^{(16)} & =(R^2+L^2)^8+56R^2L^2(R^2+L^2)^6+420R^4L^4(R^2+L^2)^4 \\[0.2cm]
    &\qquad\quad +560R^6L^6(R^2+L^2)^2+70R^8L^6, \\[0.2cm]
    S_{10}^{(18)} & =(R^2+L^2)^9+72R^2L^2(R^2+L^2)^7+756R^4L^4(R^2+L^2)^5  \\[0.2cm]
    &\qquad\quad +1680R^6L^6(R^2+L^2)^3+630R^8L^8(R^2+L^2),
\end{align*}
All three cases $n=8,9,10$ admit further simplifications.

For $P_8$ Theorem \ref{th:4.2-3} gives:
$$  d_1^2+d_5^2=d_2^2+d_6^2=d_3^2+d_7^2=d_4^2+d_8^2=2(R^2+L^2).     $$
The vertices $A_1$, $A_3$, $A_5$, $A_7$ and $A_2$, $A_4$, $A_6$, $A_8$ form two squares, so they satisfy ``additional'' cyclic relations for $P_4$.

Generally, if $n$ is divisible by 4:

\begin{theorem}\label{the:4.4-1}
For any regular $n$-gon, if $n=4p$:
\begin{align*}
    d_1^4+d_{1+p}^4+d_{1+2p}^4+d_{1+3p}^4 & =\cdots=d_p^4+d_{2p}^4+d_{3p}^4+d_{4p}^4=4\big((R^2+L^2)^2+2R^2L^2\big), \\[0.2cm]
    d_1^6+d_{1+p}^6+d_{1+2p}^6+d_{1+3p}^6 & =\cdots=d_p^6+d_{2p}^6+d_{3p}^6+d_{4p}^6=4\big((R^2+L^2)^3+6R^2L^2(R^2+L^2)\big).
\end{align*}
\end{theorem}

For $P_9$ Theorem \ref{th:4.3-1} gives:
\begin{align*}
    d_1^2+d_4^2+d_7^2=d_2^2+d_5^2+d_8^2=d_3^2+d_6^2+d_9^2 & =3(R^2+L^2), \\[0.2cm]
    d_1^4+d_4^4+d_7^4=d_2^4+d_5^4+d_8^4=d_3^4+d_6^4+d_9^4 & =3\big((R^2+L^2)^2+2R^2L^2\big).
\end{align*}
For $P_{10}$, from Theorem \ref{th:4.2-3}:
$$  d_1^2+d_6^2=d_2^2+d_7^2=d_3^2+d_8^2=d_4^2+d_9^2=d_5^2+d_{10}^2=2(R^2+L^2).      $$
The vertices $A_1$, $A_3$, $A_5$, $A_7$, $A_9$ and $A_2$, $A_4$, $A_6$, $A_8$, $A_{10}$ form two re\-gu\-lar pentagons, so they satisfy ``additional'' cyclic relations for $P_5$.

Generally, if $n$ is divisible by 5:

\begin{theorem}\label{th:4.4-2}
For any regular $n$-gon, if $n=5t$
\begin{align*}
    d_1^2+d_{1+t}^2+d_{1+2t}^2+d_{1+3t}^2+d_{1+4t}^2=\cdots & =d_t^2+d_{2t}^2+d_{3t}^2+d_{4t}^2+d_{5t}^2=5(R^2+L^2), \\[0.2cm]
    d_1^4+d_{1+t}^4+d_{1+2t}^4+d_{1+3t}^4+d_{1+4t}^4=\cdots & =d_t^4+d_{2t}^4+d_{3t}^4+d_{4t}^4+d_{5t}^4=5\big((R^2+L^2)^2+2R^2L^2\big), \\[0.2cm]
    d_1^6+d_{1+t}^6+d_{1+2t}^6+d_{1+3t}^6+d_{1+4t}^6=\cdots & =d_t^6+d_{2t}^6+d_{3t}^6+d_{4t}^6+d_{5t}^6 \\[0.2cm]
    & =5\big((R^2+L^2)^3+6R^2L^2(R^2+L^2)\big), \\[0.2cm]
    d_1^8+d_{1+t}^8+d_{1+2t}^8+d_{1+3t}^8+d_{1+4t}^8=\cdots & =d_t^8+d_{2t}^8+d_{3t}^8+d_{4t}^8+d_{5t}^8 \\[0.2cm]
    & =5\big((R^2+L^2)^4+12R^2L^2(R^2+L^2)^2+6R^4L^4\big).
\end{align*}
\end{theorem}

To summarize the obtained results, we conclude: every regular $n$-gon has an $n-1$ number of cyclic averages, but if $n$ is the composite number we have ``additional'' relations for the distances, which are obtained from the cyclic averages of the $n_1$-gon, where $n_1$ is divisible of $n$.

\bigskip
\section{Rational Distances Problem. Solution for $n=24$}
\label{sec:5}

Is there a point all of whose distances to the vertices of the unit polygon are rational? The problem has a long history especially for the case of a square. An extensive historical review is given in \cite{6, 7, 8}. For case of an equilateral triangle answer is positive \cite{9}. According to \cite{10} open problems are in following cases
$$  n=4,6,8,12\;\text{and}\;24.      $$
For $n=6$ -- only trivial point is known -- the centroid of the unit hexagon.

By Theorem \ref{th:4.3} the side $a_n$ of the regular $n$-gon is:
$$  \frac{a_n^2}{2\sin^2\frac{\pi}{n}}=S_n^{(2)}\pm\sqrt{3(S_n^{(2)})^2-2S_n^{(4)}}\,.      $$
For the unit icositetragon $(n=24)$:
$$  \sin\frac{\pi}{24}=\frac{1}{2}\,\sqrt{\frac{S_{24}^{(2)}\pm\sqrt{3(S_{24}^{(2)})^2-2S_{24}^{(4)}}}{S_{24}^{(4)}-(S_{24}^{(2)})^2}}\,.     $$
The right side is the root of the fourth degree polynomial equation with rational coefficients:
$$  8\big(S_n^{(4)}-(S_n^{(2)})^2\big)x^4-4S_n^{(2)}x^2+1=0,        $$
thus it is the algebraic number of degree $\leq 4$. On the other hand
$$  \sin\frac{\pi}{24}=\frac{1}{2}\,\sqrt{2-\sqrt{2+\sqrt{3}}}\,,     $$
is the algebraic number of degree $>4$ \cite{11}. So,

\begin{theorem}\label{th:5.1}
There is not a point in the plane that is at rational distances from the vertices of the unit regular $24$-gon.
\end{theorem}

For positive answers for the $P_4$ and $P_6$ cases the necessary conditions are the rationalities of the equal areas:
\begin{enumerate}
\item[-] $\triangle_{(d_1,\sqrt{2}\,d_2,d_3)}=\triangle_{(d_2,\sqrt{2}d_3,d_4)}$, if $n=4$;

\item[-] $\sqrt{3}\,\triangle_{(d_1,d_3,d_5)}=\sqrt{3}\,\triangle_{(d_2,d_4,d_6)}$, if $n=6$.
\end{enumerate}

\bigskip

\section{Sphere as Locus of Constant $\sum_{[R,L]}$ Sums}
\label{sec:6}

For regular polygons with different vertices the number of the $\sum_{[R,L]}$ sums are different too. As we see, unlike the plane case, dual Platonic solids have the same number of the $\sum_{[R,L]}$ sums:
\begin{align*}
    & \text{regular tetrahedron -- } \sum\nolimits_{[R,L]}^{(2)}, \;\; \sum\nolimits_{[R,L]}^{(4)}; \\
    & \text{octahedron and cube -- } \sum\nolimits_{[R,L]}^{(2)}, \;\; \sum\nolimits_{[R,L]}^{(4)}, \;\; \sum\nolimits_{[R,L]}^{(6)}; \\
    & \text{icosahedron and dodecahedron -- } \sum\nolimits_{[R,L]}^{(2)}, \;\; \sum\nolimits_{[R,L]}^{(4)}, \;\; \sum\nolimits_{[R,L]}^{(6)}, \;\;
                    \sum\nolimits_{[R,L]}^{(8)}, \;\; \sum\nolimits_{[R,L]}^{(10)}.
\end{align*}
To prove these, we consider each Platonic solid separately. In all cases, we consider solids centered at the origin and use simple Cartesian coordinates.

\subsection{Regular Tetrahedron}
\label{subsec:6.1}

The coordinates of the vertices $T_4(R)$:
$$  A_{1,2}(c,\pm c,\pm c), \;\; A_{3,4}(-c,\pm c,\mp c) \;\;\text{and}\;\; R=\sqrt{3}\,c.      $$

Consider an arbitrary point in space $M(d_1,d_2,d_3,d_4,L)$ with the coordinates: $(x,y,z)$. The distance between $M$ and the centroid $O$ of the tetrahedron:
$$  L^2=x^2+y^2+z^2.        $$
Then,
\allowdisplaybreaks
\begin{gather*}
\begin{aligned}
    d_{1,2}^2 & =(x-c)^2+(y\mp c)^2+(z\mp c)^2=R^2+L^2+2c(-x\mp y\mp z), \\[0.2cm]
    d_{3,4}^2 & =(x+c)^2+(y\mp c)^2+(z\pm c)^2=R^2+L^2+2c(x\mp y\pm z),
\end{aligned} \\[0.2cm]
\begin{aligned}
    \sum_{1}^4 d_i^4 & =\big(R^2+L^2+2c(-x\mp y\mp z)\big)^2+\big(R^2+L^2+2c(x\mp y\pm z)\big)^2 \\[0.2cm]
    & =4(R^2+L^2)^2+4c^2\big((x+y+z)^2+(-x+y+z)^2+(x-y+z)^2+(x+y-z)^2\big) \\[0.2cm]
    & =4\Big((R^2+L^2)^2+\frac{4}{3}\,R^2L^2\Big).
\end{aligned}
\end{gather*}
If for $T_4(R)$:
$$  \sum_{1}^4 d_i^4>4R^4,       $$
then

\begin{theorem}\label{th:6.1}
The locus of points in the space such that the sum of the fourth power of the distances to the vertices of a given regular tetrahedron is constant is a sphere whose center is the centroid of the tetrahedron.
\end{theorem}

\begin{remark}\label{rem:6.1}
\ \ \
\begin{enumerate}
\item[-] If $\sum\limits_{1}^4 d_i^4=4R^4$ the locus is the centroid.

\item[-] If $\sum\limits_{1}^4 d_i^4<4R^4$ the locus is the empty set.
\end{enumerate}
\end{remark}

The sums of the distances of the power more than 4 contain $x$, $y$ and $z$ (like $\alpha$ for the plane case), so for $T_4$ only the sums of the second and fourth powers are $\sum_{[R,L]}$ sums.

\subsection{Octahedron and Cube}

The coordinates of the vertices of the octahedron $T_6(R)$:
$$  A_{1,2}(\pm c,0,0), \;\; A_{3,4}(0,\pm c,0), \;\; A_{5,6}(0,0,\pm c) \;\;\text{and}\;\; R=c.        $$
For an arbitrary point $P(d_1,d_2,\dots,d_6,L)$:
\begin{align*}
    d_{1,2}^2 & =R^2+L^2\pm 2Rx, \\[0.2cm]
    d_{3,4}^2 & =R^2+L^2\pm 2Ry, \\[0.2cm]
    d_{5,6}^2 & =R^2+L^2\pm 2Rz.
\end{align*}

Beginning from $T_6$ each Platonic solid (except tetrahedron) has diametrically opposed   \linebreak      vertices, so for them Theorem \ref{th:4.2-3} is satisfied. For the sums of the fourth and sixth powers:
\allowdisplaybreaks[0]
\begin{align*}
    \sum_{1}^6 d_i^4 & =(R^2+L^2\pm 2Rx)^2+(R^2+L^2\pm 2Ry)^2+(R^2+L^2\pm 2Rz)^2 \\[0.2cm]
    & =6\Big((R^2+L^2)^2+\frac{4}{3}\,R^2L^2\Big), \\[0.2cm]
    \sum_{1}^6 d_i^6 & =6(R^2+L^2)^3+24(R^2+L^2)R^2(x^2+y^2+z^2) \\[0.2cm]
    & =6\big((R^2+L^2)^3+4R^2L^2(R^2+L^2)\big).
\end{align*}

For the cube $T_8(R)$:
\begin{gather*}
    A_{1,2}(\mp c,\mp c,\mp c), \;\; A_{3,4}(\pm c,\pm c,\mp c), \\[0.2cm]
    A_{5,6}(\pm c,\mp c,\pm c), \;\; A_{7,8}(\mp c,\pm c,\pm c)
\end{gather*}
and $R=\sqrt{3}\,c$.

The distances from the $P(d_1,d_2,\dots,d_8,L)$:
\begin{gather*}
    d_{1,2}^2=R^2+L^2\pm 2c(x+y+z), \quad d_{3,4}^2=R^2+L^2\mp 2c(x+y-z), \\[0.2cm]
    d_{5,6}^2=R^2+L^2\mp 2c(x+z-y), \quad d_{7,8}^2=R^2+L^2\pm 2c(x-y-z).
\end{gather*}

The vertices $A_1$, $A_3$, $A_5$, $A_7$ and $A_2$, $A_4$, $A_6$, $A_8$ form two regular tetrahedrons, so they satisfy the regular tetrahedron relations.

\begin{theorem}\label{th:6.2}
For an arbitrary point in the space, the sum of the quadruple of the distances to the vertices of the cube which lie on parallel faces and are endpoints of skew face diagonals, satisfies
\begin{align*}
    \sum_{1}^4 d_{2k-1}^2 & =\sum_{1}^4 d_{2k}^2=4(R^2+L^2), \\[0.2cm]
    \sum_{1}^4 d_{2k-1}^4 & =\sum_{1}^4 d_{2k}^4=4\Big((R^2+L^2)^2+\frac{4}{3}\,R^2L^2\Big).
\end{align*}
\end{theorem}

\begin{remark}\label{rem:6.2}
These quadruples do not contain the distances to diametrically opposed vertices.
\end{remark}

Thus,
\allowdisplaybreaks
\begin{align*}
    \sum_{1}^8 d_i^4 & =8\Big((R^2+L^2)^2+\frac{4}{3}\,R^2L^2\Big). \\[0.2cm]
    \sum_{1}^8 d_i^6 & =\big(R^2+L^2\pm 2c(x+y+z)\big)^3+\big(R^2+L^2\mp 2c(x+y-z)\big)^3 \\[0.15cm]
    &\qquad\quad +\big(R^2+L^2\mp 2c(x+z-y)\big)^3+\big(R^2+L^2\mp 2c(y+z-x)\big)^3 \\[0.2cm]
    & =8(R^2+L^2)^3+24(R^2+L^2)c^2\Big((x+y+z)^2 \\[0.2cm]
    &\qquad\quad\; +(x+y-z)^2+(x-y+z)^2+(x-y-z)^2\Big) \\[0.2cm]
    & =8\big((R^2+L^2)^3+4R^2L^2(R^2+L^2)\big).
\end{align*}

If for $T_6(R)$ and $T_8(R)$ is satisfied
$$  \sum_{i=1}^n d_i^{2m}>nR^{2m}, \;\; n=6,8;       $$
then

\begin{theorem}\label{th:6.3}
The locus of points in the space such that the sum of the sixth (fourth) power of distances to the vertices of a given octahedron (cube) is constant is a sphere whose center is the centroid of the octahedron (cube).
\end{theorem}

\begin{remark}\label{rem:6.3}
\ \ \
\begin{enumerate}
\item[-] If $\sum\limits_{1}^n d_i^{2m}=nR^{2m}$ the locus is the centroid.

\item[-] If $\sum\limits_{1}^n d_i^{2m}<nR^{2m}$ the locus is the empty set.
\end{enumerate}
\end{remark}

\subsection{Icosahedron and Dodecahedron}
\label{subsec:6.3}

The coordinates of the vertices of icosahedron $T_{12}(R)$:
\begin{gather*}
    A_{1,2}(0,\pm c,\pm c\varphi), \;\; A_{3,4}(0,\mp c,\pm c\varphi), \\[0.2cm]
    A_{5,6}(\pm c,\pm c\varphi,0), \;\; A_{7,8}(\pm c,\mp c\varphi,0), \\[0.2cm]
    A_{9,10}(\pm c\varphi,0,\pm c), \;\; A_{11,12}(\pm c\varphi,0,\mp c),
\end{gather*}
where $\varphi$ is the golden ratio $\varphi=\frac{1+\sqrt{5}}{2}$ and $R=c\sqrt{1+\varphi^2}\,$. For an arbitrary point   \linebreak   $P(d_1,d_2,\dots,d_{12},L)$:
\allowdisplaybreaks[0]
\begin{gather*}
    d_{1,2}^2=R^2+L^2\mp 2c(y+z\varphi), \quad d_{3,4}^2=R^2+L^2\pm 2c(y-z\varphi), \\[0.2cm]
    d_{5,6}^2=R^2+L^2\mp 2c(x+y\varphi), \quad d_{7,8}^2=R^2+L^2\mp 2c(x-y\varphi), \\[0.2cm]
    d_{9,10}^2=R^2+L^2\mp 2c(z+x\varphi), \quad d_{11,12}^2=R^2+L^2\pm 2c(z-x\varphi).
\end{gather*}
Then
\allowdisplaybreaks[0]
\begin{align*}
    \sum_{1}^{12} d_i^4 & =\sum_{1}^4 d_i^4+\sum_{5}^8 d_i^4+\sum_{9}^{12} d_i^4 \\[0.2cm]
    & =4(R^2+L^2)^2+16c^2(y^2+z^2\varphi^2)+4(R^2+L^2)^2+16c^2(x^2+y^2\varphi^2) \\[0.2cm]
    &\qquad\quad\;\, +4(R^2+L^2)^2+16c^2(z^2+x^2\varphi^2) \\[0.2cm]
    & =12(R^2+L^2)^2+16c^2(1+\varphi^2)(x^2+y^2+z^2) \\[0.2cm]
    & =12\Big((R^2+L^2)^2+\frac{4}{3}\,R^2L^2\Big). \\[0.4cm]
    \sum_{1}^{12} d_i^6 & =\sum_{1}^4 d_i^6+\sum_{5}^8 d_i^6+\sum_{9}^{12} d_i^6 \\[0.2cm]
    & =4(R^2+L^2)^3+48c^2(R^2+L^2)(y^2+z^2\varphi^2) \\[0.2cm]
    &\qquad\quad\;\, +4(R^2+L^2)^3+48c^2(R^2+L^2)(x^2+y^2\varphi^2) \\[0.2cm]
    &\qquad\quad\;\, +4(R^2+L^2)^3+48c^2(R^2+L^2)(z^2+x^2\varphi^2) \\[0.2cm]
    & =12\big((R^2+L^2)^3+4R^2L^2(R^2+L^2)\big).
\end{align*}
For the sum of the eighth power
\allowdisplaybreaks
\begin{align*}
    \sum_{1}^4 d_i^8 & =4(R^2+L^2)^4+96c^2(R^2+L^2)^2(y^2+z^2\varphi^2)+64c^4(y^4+z^4\varphi^4+6y^2z^2\varphi^2), \\[0.2cm]
    \sum_{1}^{12} d_i^8 & =12(R^2+L^2)^4+96c^2(R^2+L^2)^2(x^2+y^2+z^2)(1+\varphi^2) \\[0.2cm]
    &\qquad\qquad +64c^4\big((x^4+y^4+z^4)(1+\varphi^4)+6(x^2y^2+x^2z^2+y^2z^2)\varphi^2\big), \\[0.2cm]
\intertext{Because}
    &\qquad 1+\varphi^4=3\varphi^2 \;\;\text{and}\;\; \varphi^2=\frac{1}{5}\,(1+\varphi^2)^2, \\[0.2cm]
    \sum_{1}^{12} d_i^8 & =12\Big((R^2+L^2)^4+8R^2L^2(R^2+L^2)^2+\frac{16}{5}\,R^4L^4\Big). \\[0.2cm]
    \sum_{1}^{4} d_i^{10} & =4(R^2+L^2)^5+80(R^2+L^2)^3c^2\big((y+z\varphi)^2+(y-z\varphi)^2\big) \\[0.2cm]
    &\qquad\quad\;\, +160(R^2+L^2)c^4\big((y+z\varphi)^4+(y-z\varphi)^4\big) \\[0.2cm]
    & =4(R^2+L^2)^5+160(R^2+L^2)^3c^2(y^2+z^2\varphi^2) \\[0.2cm]
    &\qquad\quad\;\, +320(R^2+L^2)c^4(y^4+z^4\varphi^4+6y^2z^2\varphi^2). \\[0.2cm]
    \sum_{1}^{12} d_i^{10} & =12(R^2+L^2)^5+160(R^2+L^2)^3c^2(x^2+y^2+z^2)(1+\varphi^2) \\[0.2cm]
    &\qquad\qquad +320(R^2+L^2)c^4\big((1+\varphi^4)(x^4+y^4+z^4)+3\varphi^2(2x^2y^2+2x^2z^2+2y^2z^2)\big) \\[0.2cm]
    & =12\Big((R^2+L^2)^5+\frac{40}{3}\,R^2L^2(R^2+L^2)^3+16R^4L^4(R^2+L^2)\Big).
\end{align*}

Divide the vertices of the dodecahedron -- $T_{20}(R)$ into two groups, the vertices $A_1,A_2,\dots,A_8$ which form a cube and other vertices -- $A_9,A_{10},\dots,A_{20}$. Then the coordinates:
\begin{gather*}
    A_{1,2}(\mp c,\mp c,\mp c), \;\; A_{3,4}(\pm c,\pm c,\mp c), \\[0.2cm]
    A_{5,6}(\pm c,\mp c,\pm c), \;\; A_{7,8}(\mp c,\pm c,\pm c), \\[0.2cm]
    A_{9,10}\Big(0,\pm\frac{c}{\varphi},\pm c\varphi\Big), \;\; A_{11,12}\Big(0,\mp\frac{c}{\varphi},\pm c\varphi\Big), \\[0.2cm]
    A_{13,14}\Big(\pm\frac{c}{\varphi},\pm c\varphi,0\Big), \;\; A_{15,16}\Big(\mp\frac{c}{\varphi},\pm c\varphi,0\Big), \\[0.2cm]
    A_{17,18}\Big(\pm c\varphi,0,\pm\frac{c}{\varphi}\Big), \;\; A_{19,20}\Big(\pm c\varphi,0,\mp\frac{c}{\varphi}\Big).
\end{gather*}
and $R=\sqrt{3}\,c$.

Consider an arbitrary point $P(d_1,d_2,\dots,d_{20},L)$. For the distances $d_1,d_2,\dots,d_8$ we use the respective distances of the cube, and for others:
\allowdisplaybreaks
\begin{gather*}
    d_{9,10}^2=R^2+L^2\mp 2c\Big(\frac{y}{\varphi}+z\varphi\Big), \;\; d_{11,12}^2=R^2+L^2\pm 2c\Big(\frac{y}{\varphi}-z\varphi\Big), \\[0.2cm]
    d_{13,14}^2=R^2+L^2\mp 2c\Big(\frac{x}{\varphi}+y\varphi\Big), \;\; d_{15,16}^2=R^2+L^2\pm 2c\Big(\frac{x}{\varphi}-y\varphi\Big), \\[0.2cm]
    d_{17,18}^2=R^2+L^2\mp 2c\Big(\frac{z}{\varphi}+x\varphi\Big), \;\; d_{19,20}^2=R^2+L^2\pm 2c\Big(\frac{z}{\varphi}-x\varphi\Big),
\end{gather*} \\[-1cm]
\begin{align*}
    \sum_{1}^{20} d_i^4 & =8(R^2+L^2)^2+\frac{32}{3}\,R^2L^2+\sum_{9}^{20} d_i^4 \\[0.2cm]
    & =8(R^2+L^2)^2+\frac{32}{3}\,R^2L^2+12(R^2+L^2)^2+16c^2(x^2+y^2+z^2)\Big(\frac{1}{\varphi^2}+\varphi^2\Big) \\[0.2cm]
    & =20\Big((R^2+L^2)^2+\frac{4}{3}\,R^2L^2\Big). \\[0.2cm]
    \sum_{1}^{20} d_i^6 & =8(R^2+L^2)^3+32R^2L^2(R^2+L^2)+12(R^2+L^2)^3 \\[0.2cm]
    &\qquad\quad\;\, +3(R^2+L^2)16c^2(x^2+y^2+z^2)\Big(\frac{1}{\varphi^2}+\varphi^2\Big) \\[0.2cm]
    & =20\big((R^2+L^2)^3+4R^2L^2(R^2+L^2)\big). \\[0.2cm]
    \sum_{1}^8 d_i^8 & =\big(R^2+L^2\pm 2c(x+y+z)\big)^4+\big(R^2+L^2\mp 2c(x+y-z)\big)^4 \\[0.2cm]
    &\qquad +\big(R^2+L^2\mp 2c(x+z-y)\big)^4+\big(R^2+L^2\pm 2c(x-y-z)\big)^4 \\[0.2cm]
    & =8(R^2+L^2)^4+64R^2L^2(R^2+L^2)^2+\frac{64}{9}\,R^4(2L^4+8x^2y^2+8x^2z^2+8y^2z^2), \\[0.2cm]
    \sum_{9}^{20} d_i^8 & =12(R^2+L^2)^4+96(R^2+L^2)^2L^23c^2 \\[0.2cm]
    &\qquad +64c^4\Big((x^4+y^4+z^4)\Big(\varphi^4+\frac{1}{\varphi^4}\Big)+6x^2y^2+6x^2z^2+6y^2z^2\Big) \\[0.2cm]
    & =12(R^2+L^2)^4+96(R^2+L^2)^2R^2L^2 \\[0.2cm]
    &\qquad +\frac{64}{9}\,R^4\big(7(x^4+y^4+z^4)+6x^2y^2+6x^2z^2+6y^2z^2\big), \\[0.2cm]
    \sum_{1}^{20} d_i^8 & =20\Big((R^2+L^2)^4+8R^2L^2(R^2+L^2)^2+\frac{16}{5}\,R^4L^4\Big).
\end{align*}

Like $T_{12}$, maximal power for $T_{20}$ which depends on $R$ and $L$ only is 10. Indeed,
\begin{align*}
    \sum_{1}^8 d_i^{10} & =\big(R^2+L^2\pm 2c(x+y+z)\big)^5+\big(R^2+L^2\mp 2c(x+y-z)\big)^5 \\[0.2cm]
    &\qquad +\big(R^2+L^2\mp 2c(x+z-y)\big)^5+\big(R^2+L^2\pm 2c(x-y-z)\big)^5 \\[0.2cm]
    & =8(R^2+L^2)^5+320(R^2+L^2)^3c^2L^2 \\[0.2cm]
    &\qquad +320(R^2+L^2)c^4\big(2(x^4+y^4+z^4)+12(x^2y^2+x^2z^2+y^2z^2)\big), \\[0.2cm]
    \sum_{9}^{20} d_i^{10} & =12(R^2+L^2)^5+160(R^2+L^2)^3c^2L^2\Big(\frac{1}{\varphi^2}+\varphi^2\Big) \\[0.2cm]
    &\qquad +320(R^2+L^2)c^4\Big(\Big(\frac{1}{\varphi^4}+\varphi^4\Big)(x^4+y^4+z^4)+6(x^2y^2+x^2z^2+y^2z^2)\Big), \\[0.2cm]
    \sum_{1}^{20} d_i^{10} & =20\Big((R^2+L^2)^2+\frac{40}{3}\,R^2L^2(R^2+L^2)^3+16R^4L^4(R^2+L^2)\Big).
\end{align*}

If for $T_{12}(R)$ and $T_{20}(R)$ is satisfied
$$  \sum_{i=1}^n d_i^{2m}>nR^{2m}, \;\; n=12,20,        $$
then

\begin{theorem}\label{th:4.4}
The locus of points in the space such that the sum of the $2m$-th power of distances to the vertices of a given icosahedron (dodecahedron) is constant is a sphere, when
$$  m=1,2,3,4\;\text{and}\;5.       $$
The center of the sphere is the centroid of the icosahedron (dodecahedron).
\end{theorem}

\begin{remark}\label{rem:6.4}
\ \ \
\begin{enumerate}
\item[-] If $\sum\limits_{i=1}^n d_i^{2m}=nR^{2m}$ the locus is the centroid.

\item[-] If $\sum\limits_{i=1}^n d_i^{2m}<nR^{2m}$ the locus is the empty set.
\end{enumerate}
\end{remark}

\bigskip
\section{Cyclic Averages of Platonic Solids}
\label{sec:7}

Summarize the obtained results, in terms of the cyclic averages:

\begin{theorem}\label{th:7.1}
The cyclic averages of the Platonic solids are the following:
\begin{align*}
    S_{[4]}^{(2)}=S_{[6]}^{(2)}=S_{[8]}^{(2)}=S_{[12]}^{(2)}=S_{[20]}^{(2)} & =R^2+L^2, \\[0.2cm]
    S_{[4]}^{(4)}=S_{[6]}^{(4)}=S_{[8]}^{(4)}=S_{[12]}^{(4)}=S_{[20]}^{(4)} & =(R^2+L^2)^2+\frac{4}{3}\,R^2L^2, \\[0.2cm]
    S_{[6]}^{(6)}=S_{[8]}^{(6)}=S_{[12]}^{(6)}=S_{[20]}^{(6)} & =(R^2+L^2)^3+4R^2L^2(R^2+L^2), \\[0.2cm]
    S_{[12]}^{(8)}=S_{[20]}^{(8)} & =(R^2+L^2)^4+8R^2L^2(R^2+L^2)^2+\frac{16}{5}\,R^4L^4, \\[0.2cm]
    S_{[12]}^{(10)}=S_{[20]}^{(10)} & =(R^2+L^2)^5+\frac{40}{3}\,R^2L^2(R^2+L^2)^3+16R^4L^4(R^2+L^2).
\end{align*}
\end{theorem}

Eliminate $L$ and $R$ from the relations, we obtain direct relations among the cyclic averages of the Platonic solids.

\begin{theorem}\label{th:7.2}
For each Platonic solid $(n=4,6,8,12,20)$:
$$  S_{[n]}^{(4)}+\frac{16}{9}\,R^4=\Big(S_{[n]}^{(2)}+\frac{2}{3}\,R^2\Big)^2.     $$
\end{theorem}

This result for regular simplicial and regular polytopic distances is obtained in \cite{12} and \cite{13}, respectively.

\begin{theorem}\label{th:7.3}
For each Platonic solid, except the tetrahedron $(n=6,8,12,20)$:
\begin{align*}
    S_{[n]}^{(6)} & =S_{[n]}^{(2)}\big((S_{[n]}^{(2)}+2R^2)^2-8R^4\big), \\[0.2cm]
    S_{[n]}^{(6)} & =S_{[n]}^{(2)}\big(3S_{[n]}^{(4)}-2(S_{[n]}^{(2)})^2\big).
\end{align*}
\end{theorem}

\begin{theorem}\label{th:6.8}
For the icosahedron and the dodecahedron $(n=12,20)$:
\begin{align*}
    S_{[n]}^{(8)}-(S_{[n]}^{(2)})^4 & =8R^2(S_{[n]}^{(2)}-R^2)\Big((S_{[n]}^{(2)})^2+\frac{2}{5}\,R^2(S_{[n]}^{(2)}-R^2)\Big), \\[0.2cm]
    S_{[n]}^{(10)}-(S_{[n]}^{(2)})^5 & =8R^2S_{[n]}^{(2)}(S_{[n]}^{(2)}-R^2)\Big(\frac{5}{3}\,(S_{[n]}^{(2)})^2+2R^2(S_{[n]}^{(2)}-R^2)\Big), \\[0.2cm]
    S_{[n]}^{(8)} & =\frac{1}{5}\,\Big(9(S_{[n]}^{(4)})^2+12S_{[n]}^{(4)}(S_{[n]}^{(2)})^2-16(S_{[n]}^{(2)})^4\Big), \\[0.2cm]
    S_{[n]}^{(10)} & =S_{[n]}^{(2)}S_{[n]}^{(4)}\big(9S_{[n]}^{(4)}-8(S_{[n]}^{(2)})^2\big).
\end{align*}
\end{theorem}

Like the plane cases, in some space cases we have ``additional'' relations. Each Platonic solid, except the tetrahedron satisfies Theorem \ref{th:4.2-3} and for the cube and the dodecahedron Theorem~\ref{th:6.2}.

For the radius of the circumscribed sphere and the distance between the point and the centroid:

\begin{theorem}\label{th:6.9}
For each Platonic solid $(n=4,6,8,12,20)$:
\begin{align*}
    R^2 & =\frac{1}{2}\,\Big(S_{[n]}^{(2)}\pm\sqrt{4(S_{[n]}^{(2)})^2-3S_{[n]}^{(4)}}\Big), \\[0.2cm]
    L^2 & =\frac{1}{2}\,\Big(S_{[n]}^{(2)}\mp\sqrt{4(S_{[n]}^{(2)})^2-3S_{[n]}^{(4)}}\Big).
\end{align*}
\end{theorem}

The points on the circumscribed sphere satisfy
$$  4(S_{[n]}^{(2)})^2=3S_{[n]}^{(4)},        $$
so

\begin{theorem}\label{th:7.5}
For any point on the circumscribed sphere of each Platonic solid $(n\!=\!4,6,8,12,20)$:
$$  4\Big(\sum_{i=1}^n d_i^2\Big)^2=3n\sum_{i=1}^n d_i^4.       $$
\end{theorem}

\bigskip
\section{Conclusion}
\label{sec:8}

In the present paper, we introduce the $\sum_{[R,L]}$ sums and define the cyclic averages of the regular polygons and the Platonic solids. We prove the main property of the cyclic averages -- the equality of them for various regular polygons and Platonic solids. By means of the cyclic averages the distances of an arbitrary point to the vertices of the regular polygons (the plane case) and the Platonic solids (the space case) are investigated. All cases of constant sum of like powers of the distances, when the locus is a circle (a sphere), are found. General metrical relations for regular polygons (Platonic solids), which were known in special cases only, are established. Rational distances problem solved for the $n=24$ case.

\bigskip
\subsection*{Acknowledgement}

The author would like to thank Georgia Young Scientists Union (GYSU). This research was funded by GYSU. Contract no.~2506-19.

\vskip+0.5cm

\noindent \textbf{Author's address:}

\medskip

\noindent Department of Mathematics, Georgian-American High School, 18 Chkondideli Str., Tbilisi 0180, Georgia.

\noindent E-mail: \textit{mathmamuka@gmail.com}

\end{document}